\documentclass[a4paper, 12pt]{article}

\pdfoutput=1

\usepackage{color}
\usepackage{graphicx}
\usepackage{amsmath}
\usepackage{amsfonts}
\usepackage{amssymb}
\usepackage{amsthm}
\usepackage{epstopdf}
\usepackage{algorithmic}
\usepackage{setspace}
\usepackage[cm]{fullpage}

\newtheorem{remark}{\bf Remark}

\newtheorem{lemma}{\bf Lemma}

\newtheorem{corollary}{\bf Corollary}
\newtheorem{theorem}{\bf Theorem}

\numberwithin{equation}{section}
\numberwithin{definition}{section}
\numberwithin{lemma}{section}
\numberwithin{remark}{section}
\numberwithin{corollary}{section}
\numberwithin{theorem}{section}
\numberwithin{example}{section}

\begin{document}

\title{A new incremental method of computing the limit load in deformation plasticity models}
\author{J. Haslinger$^{1,2}$, S. Repin$^{3,4}$, S. Sysala$^1$\\ \\
$^1$Institute of Geonics CAS, Ostrava, Czech Republic\\
$^2$Charles University in Prague, Prague, Czech Republic\\
$^3$St. Petersburg   Department   of   V.A. Steklov    Institute   of   Mathematics\\
of   the   Russian   Academy   of   Sciences, Russia\\
$^4$University of Jyv\"{a}skyl\"a, Finland}

\maketitle

\begin{abstract}
The aim of this paper is to introduce a new incremental procedure that can be used for numerical evaluation of the limit load. Existing incremental type methods are based on parametrization of the energy by the loading parameter $\zeta\in[0,\zeta_{lim})$, where $\zeta_{lim}$ is generally unknown. In the new method, the incremental procedure is operated in terms of an inverse mapping and the respective parameter $\alpha$ is changing in the interval $(0,+\infty)$. Theoretically, in each step of this algorithm, we obtain a guaranteed lower bound of $\zeta_{lim}$. Reduction of the problem to a finite element subspace associated with a mesh $\mathcal T_h$ generates computable bound $\zeta_{lim,h}$. Under certain assumptions, we prove that $\zeta_{lim,h}$ tends to $\zeta_{lim}$ as $h\rightarrow0_+$. Numerical tests confirm practical efficiency of the suggested method.
\end{abstract}

\section{Introduction}\label{sec_int}
\setcounter{equation}{0}

Elastic-perfectly plastic models belong among fundamental nonlinear models which are useful for estimation of yield strengths or failure zones in bodies caused by applied forces. Such models are mostly quasistatic (see, e.g., \cite{NPO08, HaRe99, NeHl81}) to catch the unloading phenomenon. Since we are only interested in monotone loading processes, this phenomenon can be neglected and the class of models based on the deformation theory of plasticity is adequate (see, e.g., \cite{J09, M77, NeHl81, RS95, T85}). The Hencky model associated with the von Mises yield criterion belongs to this class as well as other models with different yield conditions.  Each model from this class leads to a static problem for a given load functional $L$ representing the work of surface or volume forces. The problem can be formulated both in terms of stresses or displacements. These two approaches generate a couple of mutually dual problems. 

The variational problem formulated in terms of stresses leads to minimization of a strictly convex, quadratic functional on the set of statically and plastically admissible stress fields. On the other hand, the stored energy functional appearing in the variational problem for displacements has only a linear growth at infinity with respect to the strain tensor or some components of this tensor. Existence of a finite limit load reflects specifics of this class of problems. Unlike other problems in continuum mechanics with superlinear growth of energy, exceeding of the limit load leads to absence of a solution satisfying the equilibrium equations and constitutive relations. Physically this means that under this load the body cannot exists as a consolidated object. Therefore, finding limit loads is an important problem in the theory of elasto--plastic materials and other close problems.

The limit load has the form $\zeta_{lim}L$, where $\zeta_{lim}$ is a nonnegative scalar parameter. In particular, no solution exists for the load $\zeta L$ with $\zeta>\zeta_{lim}$. If we use the static approach to the limit load analysis then $\zeta_{lim}$ is defined as the largest value of $\zeta\geq0$ for which there exists a statically and simultaneously plastically admissible stress field. The value $\zeta_{lim}$ is generally unknown and its finding is an important question in analysis of elasto-plastic structures.

From the above mentioned facts, it is natural to consider the set $\{\zeta L\ |\; \zeta\in\mathbb R_+ \}$ of loads and parametrize the problem in this way. Besides $\zeta_{lim}$, one can detect other interesting thresholds on the loading path that represent global material response, namely, $\zeta_e$ - the end of elasticity and $\zeta_{prop}$ - the limit of proportionality.  For $\zeta\leq\zeta_{e}$, the response is purely elastic (linear) and for $\zeta\in[\zeta_{prop},\zeta_{lim}]$, the response is strongly nonlinear.

There are two possibilities how to evaluate $\zeta_{lim}$: (a) compute $\zeta_{lim}$ directly and (b) find a suitable adaptive strategy that enlarges $\zeta$ and detects the unknown value $\zeta_{lim}$. 

 In the first group of methods, the problem of finding the limit load is reduced to a special variational problem in terms of displacements (kinematic approach), which involves an isoperimetric condition $L(v)~=~1$. For example, the respective problem of limit analysis for the classical Hencky model with the von Mises condition reads as follows:
$$\zeta_{lim}=\inf\limits_{{v\in V,\,L(v)=1}\atop
{\rm div}=0}\,\int_\Omega|\varepsilon(v)|\,dx,$$
where $\mathbb V$ is a subspace of $H^1(\Omega;\mathbb R^3)$ of functions vanishing on the Dirichlet part of the boundary (see notation of Section \ref{sec_setting}).
However, this problem is not simple for numerical analysis because it is related to a nondifferentiable functional and contains the divergence free constraint. The respective numerical approaches developed to overcome these difficulties often use saddle point formulations with augmented Lagrangians (see, e.g., \cite{CG10, Ch96}). Other methods use techniques developed for minimization of nondifferentiable functionals.

Engineering computations often use rather straightforward incremental approach based on solving the displacement problem for a sequence $\zeta_{k+1}=\zeta_k+\triangle\zeta$, $k=0,1,\ldots$. However, this may be costly without a suitable adaptive strategy for the increment $\triangle\zeta$. Moreover, since the problems are solved numerically on a finite dimensional subspace, it is difficult to
reliably verify that $\zeta_k+\triangle\zeta$ exceeds $\zeta_{lim}$.

Solving the problem for a fixed value of $\zeta$ in terms of stresses leads to a variational problem, which is very difficult from the computational point of view since it requires an approximation of equilibrated stress fields which satisfy the plasticity condition in the pointwise sense. Therefore, numerical approaches are usually applied to variational formulations in terms of displacements in spite of the fact that these problems are more complicated from the theoretical point of view because the basic variational problem may have no solution in a standard Sobolev space.

In this paper, we suggest a different incremental technique which is based on transforming the problem into a dual form. Instead of $\zeta$, we use another parameter $\alpha\geq0$ that is dual to $\zeta$ and such that $\zeta\rightarrow\zeta_{lim}^-$ (i.e. $\zeta\rightarrow\zeta_{lim}$ from the left) as $\alpha\rightarrow+\infty$. For a given value of $\alpha$, we derive a minimization problem for the stored strain energy functional subject to the constraint $L(v)=\alpha$ whose solutions define a unique way the respective value $\zeta:=\zeta(\alpha)$. Using this approach, it is possible to find a loading path associated with the given load $L$ which provides a valuable information about $\zeta_e$, $\zeta_{prop}$ and $\zeta_{lim}$.

The parameter $\alpha$ was originally introduced in \cite{CHKS14,SHHC13} for a discrete version of the Hencky problem using the formulation of the problem in terms of displacements (primal problem). The goal of this paper is to generalize this idea to the continuous setting. This generalization however is not straightforward owing to the fact that the primal formulation is not well-posed on classical Sobolev spaces. Therefore the dual formulation of the problem will be used for finding the mutual relation between $\zeta$ and $\alpha$. 

The paper is organized as follows: In Section \ref{sec_setting}, we introduce basic notation, define elasto-plastic problems, and recall some results concerning properties of solutions. In Section \ref{sec_parameters}, the loading parameters $\zeta$ and $\alpha$ are introduced. Then the function $\psi: \alpha\mapsto\zeta$  is constructed and its properties are established. In Section \ref{sec_problems_alpha}, we formulate problems in terms of stresses and displacements related to a prescribed value of $\alpha$. Section \ref{sec_discretization} is devoted to standard finite element discretizations of the problems and to convergence analysis. Finally, in Section \ref{sec_eval}, we present two examples with different yield functions which confirm  practical efficiency of the suggested method.

\section{Elastic-perfectly plastic problem based on the deformation theory of plasticity}
\label{sec_setting}

We consider an elasto-plastic body occupying  a bounded domain $\Omega\subseteq \mathbb{R}^3$ with Lipschitz boundary $\partial\Omega$. It is assumed that $\partial \Omega = \overline{\Gamma}_D \cup \overline{\Gamma}_N$, where $\Gamma_D$ and $\Gamma_N$ are open and disjoint sets, $\Gamma_D$ has a positive surface measure. Surface tractions of density $f$ are applied on $\Gamma_N$ and the body is subject to a volume force $F$.

For the sake of simplicity, we assume that the material is homogeneous. Then, the generalized Hooke's law is represented by the tensor $C$, which does not depend on $x\in \Omega$ and satisfies the following conditions of symmetry and positivity:
$$
\begin{array}{c}
C\eta\in \mathbb{R}_{sym}^{3\times 3}\qquad\forall \eta\in\mathbb{R}_{sym}^{3\times 3},\\[1mm]
C\eta:\xi=\eta:C\xi\qquad\forall \eta,\xi\in\mathbb{R}_{sym}^{3\times 3},\\[1mm]
\exists \delta>0:\;\;\;C\eta:\eta\geq \delta(\eta:\eta)\qquad\forall \eta\in\mathbb{R}_{sym}^{3\times 3},
\end{array}
$$
where $\mathbb{R}_{sym}^{3\times 3}$ is the space of all symmetric, $(3\times 3)$ matrices and $\eta:\xi=\eta_{ij}\xi_{ij}$ denotes the scalar product on $\mathbb{R}_{sym}^{3\times 3}$.

By $S := L^2(\Omega;\mathbb{R}^{3\times 3}_{sym})$, we denote the set of symmetric tensor valued functions with square summable coefficients representing stress and strain fields. On $S$, we define the scalar product
$$\langle \tau, e\rangle = \int_{\Omega}{\tau:e}\,{\mbox{d}}x,\quad \tau, e\in S,$$
and the respective norm $\|\tau\|=\langle\tau, \tau\rangle^{1/2}$. Also, we use equivalent norms suitable for stress ($\tau$) and strain ($e$) fields, respectively:
$$\|\tau\|_{C^{-1}}:=\langle C^{-1}\tau, \tau\rangle^{1/2},\quad \|e\|_{C}=\langle Ce, e\rangle^{1/2}.$$
Further, let
$$\mathbb{V}:= \left\{ v\in H^1(\Omega;\mathbb R^3)\ |\;\; v = 0\ \mbox{on }\Gamma_D \right\}$$
denote the space of kinematically admissible displacements and
$$L(v):= \int_{\Omega} F{\cdot}v {\mbox{d}}x + \int_{\Gamma_N} f{\cdot}v {\mbox{d}}s,\quad v\in\mathbb V$$
be the load functional. We assume that
$$\begin{array}{l l}
(L_1) &  F\in L^2(\Omega;\mathbb R^3),\;\;f\in L^2(\Gamma_N;\mathbb R^3),\\[2mm]
(L_2) &  \|F\|_{ L^2(\Omega;\mathbb R^3)}+\|f\|_{L^2(\Gamma_N;\mathbb R^3)}>0.
\end{array}
$$

The following closed, convex sets represent statically and plastically admissible stress fields, respectively:
$$\Lambda_{L} := \left\{ \tau\in S\ |\ \langle \tau, \varepsilon(v) \rangle = L(v)\quad \forall v\in \mathbb V \right\},$$
$$
  P:=\left\{ \tau \in S\ | \ \Phi(\tau(x)) \leq \gamma \quad \mbox{for a. a. } x\in \Omega\right\}.
$$
Here, $\Phi:\ \mathbb{R}_{sym}^{3\times 3} \rightarrow \mathbb{R}$ is a continuous, convex yield function such that $\Phi(0) = 0$, $\gamma>0$ represents the initial yield stress (which is constant in $\Omega$ due to the homogeneity assumption) and $\varepsilon(v)=\frac{1}{2}\left(\nabla v+(\nabla v)^T\right)$ is the linearized strain tensor corresponding to the displacement $v$.

In accordance with the Haar-Karman variational principle, the actual stress
is a minimizer of the variational problem:
\begin{equation*}
(\mathcal{P}^{*})\qquad
   \mbox{find }\sigma\in\Lambda_{L} \cap P: \quad \mathcal{I}(\sigma) \leq \mathcal{I}(\tau)\quad \forall\tau \in \Lambda_{L} \cap P,
\end{equation*}
where
$$\mathcal{I}(\tau):= \frac{1}{2}\|\tau\|_{C^{-1}}^2, \quad \tau\in S.$$
Problem $(\mathcal{P}^{*})$ has a unique solution if and only if $\Lambda_{L}\cap P \not= \emptyset$.

The corresponding dual problem is formulated in terms of displacements.
It has the form:
\begin{equation*}
(\mathcal{P})\qquad
   \mbox{find } u\in\mathbb V: \quad J(u)\leq J(v)\quad\forall v\in \mathbb V,
\end{equation*}
where
$$J(v):=\Psi(\varepsilon(v))-L(v),\quad v\in \mathbb V, $$
\begin{equation}
  \Psi(e) := \sup_{\tau\in P} \left\{ \langle \tau, e\rangle -\frac{1}{2}\|\tau\|_{C^{-1}}^2 \right\}=-\frac{1}{2} \| \Sigma(e)\|_{C^{-1}}^2 + \langle \Sigma(e), e \rangle \qquad \forall e\in S
\label{Psi}
\end{equation}
and $\Sigma: S\rightarrow S$ is defined by $\Sigma(e)=\Pi(Ce)$ for any $e\in S$. Here $\Pi$ denotes the projection of $S$ on $P$ with respect to the scalar product $\langle C^{-1}\sigma,\tau\rangle$. In addition, $\Sigma$ is the Fr\'echet derivative of $\Psi$, i.e. $\Sigma(e)=\mathbb D\Psi(e)$ for any $e\in S$.  The functional $\Psi$ is convex and differentiable but has only a linear growth at infinity. Therefore, existence of a solution to $(\mathcal{P})$ is not guaranteed in $\mathbb V$ or other Sobolev spaces. 

If $\Lambda_{L}\cap P \not= \emptyset$ then $(\mathcal{P})$ and $(\mathcal{P}^*)$ have finite infima and the duality relation
\begin{equation}
\inf_{v\in \mathbb V}J(v)=\sup_{\tau\in\Lambda_L\cap P}\{-\mathcal{I}(\tau)\}.
\label{duality}
\end{equation}
holds. If $(\mathcal P)$ has a solution $u$ then it satisfies the variational equation
\begin{equation}
\langle\sigma,\varepsilon(v)\rangle=L(v) \quad\forall v\in\mathbb V,
\label{eqn_P}
\end{equation}
where $\sigma:=\Sigma(\varepsilon(u))$ is the unique solution to $(\mathcal{P}^*)$.

\begin{remark}
\emph{In the special case, $P=S$, the problems $(\mathcal{P})$ and $(\mathcal{P}^*)$ lead to well-known primal
and dual formulations of elasticity problems:
\begin{equation*}
(\mathcal{P}_{e})\qquad
   \mbox{find } u_e\in\mathbb V: \quad J_e(u_e)\leq J_e(v)\quad\forall v\in \mathbb V,
\end{equation*}
where
$$J_e(v):=\frac{1}{2}\|\varepsilon(v)\|_C^2-L(v),\quad v\in \mathbb V,$$
and
\begin{equation*}
(\mathcal{P}^{*}_{e})\qquad
  \begin{array}{c}
   \mbox{find }\sigma_e\in\Lambda_{L}: \quad \mathcal{I}(\sigma_e) \leq \mathcal{I}(\tau) \quad \forall\tau \in \Lambda_{L}\;\;\mbox{(Castigliano's principle)},
  \end{array}
\end{equation*}
respectively. Both problems have unique solutions and $C\varepsilon(u_e)=\sigma_e$. Notice that if $C\varepsilon(u_e)\in P$ then $\Sigma(\varepsilon(u_e))=C\varepsilon(u_e)$ and $u_e$ also solves $(\mathcal{P})$. }
\label{remark1}
\end{remark}

\section{Parametrization of the problem}
\label{sec_parameters}

Problems $(\mathcal P)$ and $(\mathcal P^*)$ are defined for a prescribed load functional $L$. Henceforth, we consider a one parametric
family of loads $\zeta L$, where $\zeta\in\mathbb R_+$. Therefore, we use notation $(\mathcal P)_{\zeta}$, $(\mathcal P^*)_{\zeta}$, $(\mathcal P_e)_{\zeta}$, $(\mathcal P_e^*)_{\zeta}$, $\Lambda_{\zeta L}$, and $J_{\zeta}$ instead of $(\mathcal P)$, $(\mathcal P^*)$, $(\mathcal P_e)$, $(\mathcal P_e^*)$, $\Lambda_{L}$, and $J$, respectively. 

The limit load parameter $\zeta_{lim}$ is defined by
$$\zeta_{lim}:=\sup\mathcal D,\quad \mathcal D:=\{\zeta\in\mathbb R_+\ |\;  \Lambda_{\zeta L}\cap P \not= \emptyset\}.$$
Notice that, in some cases, $\zeta_{lim}$ may be infinite. However, in the majority of cases, the value of $\zeta_{lim}$ is finite. From now on, we assume that
$$(L3)\qquad \zeta_{lim}>0.$$

Problem $(\mathcal P^*)_{\zeta}$ has a unique solution for any $\zeta\in\mathcal D$. Depending on the definition of the yield function $\Phi$, we may have one of the following two situations:
\begin{equation}
(a)\;\; \mathcal D =[0,\zeta_{lim})\quad\mbox{or}\quad (b)\;\;\mathcal D=[0,\zeta_{lim}].
\label{lim_prop}
\end{equation}
In general, it is not known, whether $\zeta_{lim}\in\mathcal D$, i.e. $\Lambda_{\zeta_{lim} L}\cap P \not= \emptyset$. This is true, for example, for the von Mises or Tresca criterion  (see \cite{T85}).

From the practical point of view it is very important to know the value of $\zeta_{lim}$. The related problem of limit analysis has been considered in \cite{Ch96, RS95, T85} and publications cited therein. This minimization problem can be solved independently of the original plasticity problem by various numerical methods (see, e.g., \cite{CG10, Ch96}). However, solving this problem leads to rather complicated numerical procedures.

The aim of this paper is to propose and justify a robust way of finding $\zeta_{lim}$, which is based on a different loading parameter. The first principal idea is to introduce a nonnegative function $\phi:\mathbb R\rightarrow\overline{\mathbb R}$ as follows:
\begin{equation}
\phi(\zeta)=\left\{
\begin{array}{c c}
\mathcal{I}(\sigma(\zeta)), & \zeta\in\mathcal D,\\[1mm]
+\infty, &\mbox{otherwise.}
\end{array}\right.
\label{phi}
\end{equation}
Here, $\sigma:=\sigma(\zeta)$ denotes the unique solution to $(\mathcal P^*)_\zeta$. Properties of $\phi$ are summarized in the following lemma.
\begin{lemma}
Let the assumptions $(L1)-(L3)$ be satisfied and let $\phi:\mathbb R\rightarrow\overline{\mathbb R}$ be defined by (\ref{phi}). Then, $\phi$ is a nonnegative, strictly convex and increasing function in $\mathcal D$. Moreover,
\begin{equation}
\phi(\zeta_0)\leq\left(\frac{\zeta_0}{\zeta_1} \right)^2 \phi (\zeta_1)\quad\forall\zeta_0,\zeta_1\in\mathcal D,\;\zeta_0<\zeta_1.
\label{phi_prop}
\end{equation}
\label{lemma_phi_prop}
\end{lemma}

\begin{proof}
Let $\zeta_0,\zeta_1$ be as in (\ref{phi_prop}) and
 $$\zeta_{\lambda}:= (1-\lambda)\zeta_0 + \lambda \zeta_1, \qquad \lambda \in [0,1].$$
Then
$(1-\lambda)\sigma(\zeta_0) + \lambda \sigma(\zeta_1)  \in \Lambda_{\zeta_{\lambda} L}\cap  P,$
where $\sigma(\zeta_\theta)$ denotes the solution to $(\mathcal P^*)_{\zeta_\theta}$, $\theta \in [0,1]$.
Consequently,
\begin{equation}
\mathcal{I}(\sigma(\zeta_{\lambda})) \leq \mathcal{I} ((1- \lambda) \sigma(\zeta_0) + \lambda \sigma(\zeta_1))\leq(1- \lambda) \mathcal{I} (\sigma(\zeta_0)) + \lambda\mathcal{I}(\sigma(\zeta_1)).
\label{proof_strict_convex}
\end{equation}
Notice that the strict inequality holds in (\ref{proof_strict_convex}) for $\lambda\in(0,1)$ as $\sigma(\zeta_0)\neq\sigma(\zeta_1)$ in view of the assumption $(L2)$. Thus,  $\phi$ is convex on $\mathbb R$ and strictly convex on $\mathcal D$.

From the definition of the yield function $\Phi$, it follows that $\frac{\zeta_0}{\zeta_1} \sigma (\zeta_1) \in \Lambda_{\zeta_0 L}\cap P$. Therefore, we have:
$$\phi(\zeta_0) = \mathcal{I}(\sigma(\zeta_0)) \leq \mathcal{I}\left(\frac{\zeta_0}{\zeta_1}\sigma(\zeta_1)\right) =\left(\frac{\zeta_0}{\zeta_1} \right)^2 \phi (\zeta_1).$$
Hence, (\ref{phi_prop}) is proved and, since, $\sigma(\zeta_1)\neq0$ we conclude that $\phi$ is an increasing function on $\mathcal D$. 
\end{proof}

\begin{lemma}
Let $\zeta_{lim}\not\in\mathcal D$. Then
\begin{equation}
\lim_{\zeta\rightarrow\zeta_{lim}^-}\phi(\zeta)=+\infty.
\label{phi_prop2}
\end{equation}
\label{lemma_phi_prop2}
\end{lemma}

\begin{proof}
If $\zeta_{lim}=+\infty$ then (\ref{phi_prop2}) follows from (\ref{phi_prop}). Let $$\zeta_{lim}<+\infty$$ and suppose that $\lim_{\zeta\rightarrow\zeta_{lim}^-}\phi(\zeta)\in\mathbb R_+^1$. Then there exist sequences $\{\zeta_j\}$, $\{\sigma(\zeta_j)\}$ and an element $\bar\sigma\in S$ such that
$$\zeta_j\rightarrow\zeta_{lim}^-,\quad \sigma(\zeta_j)\rightharpoonup\bar\sigma\;\;\mbox{in }S,\quad j\rightarrow+\infty.$$
In addition, $\bar\sigma\in\Lambda_{\zeta_{lim}L}\cap P$ which contradicts the assumption.
\end{proof}

\begin{lemma}
The function $\phi$ defined by (\ref{phi}) is continuous in $\mathcal D$.
\label{lemma_phi_prop3}
\end{lemma}

\begin{proof}
Continuity of $\phi$ in $\mbox{int}\, \mathcal D$ follows from its convexity. From (\ref{phi_prop}), we see that
$$\lim_{\zeta_0\rightarrow0_+}\phi(\zeta_0)=0=\mathcal I(\sigma(0))=\phi(0).$$
Let $\zeta_{lim}\in\mathcal D$ and
$$\lim_{\zeta\rightarrow\zeta_{lim}^-}\phi(\zeta)=c\in\mathbb R_+.$$
To show that $c=\phi(\zeta_{lim})$ we proceed as in Lemma \ref{lemma_phi_prop2}. Let $\zeta_j\rightarrow\zeta_{lim}^-$ and $\sigma(\zeta_j)\rightharpoonup\bar\sigma\in\Lambda_{\zeta_{lim}L}\cap P$. Let $\tau\in\Lambda_{\zeta_{lim}L}\cap P$ be arbitrary and set $\tau_j=\frac{\zeta_j}{\zeta_{lim}}\tau\in\Lambda_{\zeta_{j}L}\cap P$. Then $\tau_j\rightarrow\tau$ in $S$ and from the definition of $(\mathcal P)_{\zeta_j}^*$ we have
$$\phi(\zeta_j)=\mathcal I(\sigma(\zeta_j))\leq\mathcal I(\tau_j).$$
Hence,
$$\mathcal I(\bar\sigma)\leq\liminf_{j\rightarrow+\infty}\mathcal I(\sigma(\zeta_j))\leq\lim_{j\rightarrow+\infty}\mathcal I(\tau_j)=\mathcal I(\tau),$$
i.e. $\bar\sigma=\sigma(\zeta_{lim})$ proving that $\phi(\zeta_{lim})=\mathcal I(\sigma(\zeta_{lim}))\leq c$. The opposite inequality $\phi(\zeta_{lim})\geq c$ follows from monotonicity of $\phi$.
\end{proof}

\begin{remark}
\emph{It is worth noting that:
\begin{itemize}
\item[$a)$]  $\phi(\zeta)=\zeta^2\mathcal{I}(\sigma_e)$ if $\zeta\in[0,\zeta_e]$, where
$$\zeta_e:=\sup\{\zeta\in\mathbb R_+\ |\;\; \zeta C\varepsilon(u_e)\in P\},\quad u_e\;\mbox{solves }(\mathcal P_e).$$
\item[$b)$] (\ref{phi_prop}) ensures a quadratic growth of $\phi$ at infinity if $\zeta_{lim}=+\infty$. 
\end{itemize}}
\label{remark3}
\end{remark}

Now, we introduce a new parameter $\alpha$, which plays a crucial role in forthcoming analysis. We set
\begin{equation}
\left.
\begin{array}{l c l}
\alpha=0 &\mbox{if}& \zeta=0,\\
\alpha\in\partial\phi(\zeta) &\mbox{if}& \zeta\in\mathcal D\setminus\{0\}
\end{array}\right\}
\label{alpha_def}
\end{equation}
From monotonicity of $\phi$, it follows that $\partial\phi(\zeta)\subset(0,+\infty)$  for any $\zeta\in\mathcal D\setminus\{0\}$. Moreover,
\begin{equation}
\bigcup_{\zeta\in\mathcal D\setminus\{0\}}\partial\phi(\zeta)=(0,+\infty).
\label{cup_gradient}
\end{equation}
 Indeed, from the definition of the subgradient of $\phi$ at $\zeta$ we know that $\alpha\in\partial\phi(\zeta)$ if and only if 
\begin{equation}
\phi(\zeta)-\alpha\zeta\leq\phi(\tilde\zeta)-\alpha\tilde\zeta\quad \forall\tilde\zeta\in\mathcal D.
\label{zeta_min}
\end{equation}
From Lemma \ref{lemma_phi_prop} - \ref{lemma_phi_prop3} and Remark \ref{remark3}$\,b)$ we know that the function $\tilde\zeta\mapsto\phi(\tilde\zeta)-\alpha\tilde\zeta$ has a unique minimizer $\zeta$ in $\mathcal D$ for any $\alpha\in[0,+\infty)$ so that (\ref{cup_gradient}) holds. This fact enables us to define the function $\psi:\mathbb R_+\rightarrow\mathcal D$, $\psi:\alpha\mapsto\zeta$, where $\zeta=\psi(\alpha)\in\mathcal D$ is the unique solution of (\ref{zeta_min}) for given $\alpha$. In the next theorem, we establish some useful properties of $\psi$.

\begin{theorem}
Let the assumptions $(L1)-(L3)$ be satisfied. Then
\begin{enumerate}
\item[(i)] $\psi$ is continuous and nondecreasing in $\mathbb R_+$;
\item[(ii)] $\psi(\alpha)\rightarrow\zeta_{lim}$ as $\alpha\rightarrow+\infty$.
\end{enumerate}
\label{theorem_psi}
\end{theorem}

\begin{proof}
Let $\alpha>0$ be given and $\phi^*$ be the Legendre-Fenchel transformation of $\phi$. It is well known that $\phi^*$ is a convex function in $\mathbb R_+$ and (\ref{alpha_def})$_2$ holds if and only if $\zeta\in\partial\phi^*(\alpha)$. Since $\zeta=\psi(\alpha)$, it holds that $\partial\phi^*(\alpha)$ is singleton and $(\phi^*)'(\alpha)=\psi(\alpha)\geq0$. Therefore, convexity and differentiability of $\phi^*$  in $\mathbb R_+$ entail that $\psi$ is continuous and  nondecreasing in $\mathbb R_+$ and $(i)$ holds.

By $(i)$, there exists $\zeta_{max}\leq\zeta_{lim}$ such that $\zeta_{max}=\lim_{\alpha\rightarrow+\infty}\psi(\alpha)$. Suppose that 
\begin{equation}
\lim_{\alpha\rightarrow+\infty}\psi(\alpha)=\zeta_{max}<\zeta_{lim}.
\label{assump_contr}
\end{equation}
Then $\phi(\psi(.))$ is bounded on $\mathbb R_+$ and 
$$\lim_{\alpha\rightarrow+\infty}\psi(\alpha)=\lim_{\alpha\rightarrow+\infty}\left\{\psi(\alpha)-\frac{\phi(\psi(\alpha))}{\alpha}\right\}\stackrel{(\ref{zeta_min})}{=}\lim_{\alpha\rightarrow+\infty}\sup_{\tilde\zeta\in\mathcal D}\left\{\tilde\zeta-\frac{\phi(\tilde\zeta)}{\alpha}\right\}\geq\lim_{\alpha\rightarrow+\infty}\left\{\hat\zeta-\frac{\phi(\hat\zeta)}{\alpha}\right\}=\hat\zeta$$
holds for any $\hat\zeta\in\mathcal D$. The choice $\hat\zeta\in(\zeta_{max},\zeta_{lim})$ contradicts (\ref{assump_contr}) and thus $(ii)$ holds.
\end{proof}

\begin{remark}
\emph{It is easy to show that $$\psi(\alpha)=\frac{1}{2\mathcal{I}(\sigma_e)}\alpha\qquad\forall\alpha\in[0,\alpha_e],$$
where $\alpha_e=2\zeta_e\mathcal{I}(\sigma_e)$, $\sigma_e$ solves $(\mathcal P^*_e)$ and $\zeta_e$ is the same as in Remark \ref{remark3}.}
\end{remark}

Figure \ref{fig_schemes} depicts three possible cases of the behaviour of $\phi$, $\psi$ for $\zeta\rightarrow\zeta_{lim}$, and $\alpha\rightarrow+\infty$, respectively.
\begin{figure}[htb!]
\centering

\begin{picture}(100,90)
\put(0,0){\vector(0,1){90}} \put(0,0){\vector(1,0){100}}
{\thicklines
\qbezier(40,10)(60,20)(59,90)
\qbezier(0,0)(25,0)(40,10)
}
\multiput(60,0)(0,4){23}{\put(0,0){\line(0,2){2}}}
\put(-4,88){\makebox(0,0)[r]{$\phi(\zeta)$}}
\put(58,-4){\makebox(0,0)[t]{$\zeta_{lim}$}}
\put(98,-4){\makebox(0,0)[t]{$\zeta$}}
\put(-30,50){\makebox(0,0)[r]{$(a)$}}
\end{picture}
\hspace*{40pt}
\begin{picture}(140,90)
\put(0,0){\vector(0,1){90}} \put(0,0){\vector(1,0){140}}
{\thicklines
\put(0,0){\line(1,2){20}}
\qbezier(20,40)(30,70)(140,70)
}
\multiput(0,73)(4,0){35}{\put(0,0){\line(1,0){2}}}

\put(-4,88){\makebox(0,0)[r]{$\psi(\alpha)$}}
\put(-2,72){\makebox(0,0)[r]{$\zeta_{lim}$}}
\put(138,-4){\makebox(0,0)[t]{$\alpha$}}
\end{picture}

\bigskip\bigskip

\begin{picture}(100,90)
\put(0,0){\vector(0,1){90}} \put(0,0){\vector(1,0){100}}
{\thicklines
\qbezier(0,0)(50,0)(60,80)
}
\multiput(60,0)(0,4){23}{\put(0,0){\line(0,2){2}}}
\put(61,80){\circle*{4}}
\put(-4,88){\makebox(0,0)[r]{$\phi(\zeta)$}}
\put(58,-4){\makebox(0,0)[t]{$\zeta_{lim}$}}
\put(98,-4){\makebox(0,0)[t]{$\zeta$}}
\put(-30,50){\makebox(0,0)[r]{$(b)$}}
\end{picture}
\hspace*{40pt}
\begin{picture}(140,90)
\put(0,0){\vector(0,1){90}} \put(0,0){\vector(1,0){140}}
{\thicklines
\put(0,0){\line(1,2){20}}
\qbezier(20,40)(30,70)(100,73)
\put(100,73){\line(1,0){40}}
}
\multiput(0,73)(4,0){35}{\put(0,0){\line(1,0){2}}}

\put(-4,88){\makebox(0,0)[r]{$\psi(\alpha)$}}
\put(-2,72){\makebox(0,0)[r]{$\zeta_{lim}$}}
\put(138,-4){\makebox(0,0)[t]{$\alpha$}}
\end{picture}

\bigskip\bigskip

\begin{picture}(100,90)
\put(0,0){\vector(0,1){90}} \put(0,0){\vector(1,0){100}}
{\thicklines
\qbezier(0,0)(80,0)(100,90)
}
\put(-4,88){\makebox(0,0)[r]{$\phi(\zeta)$}}
\put(98,-4){\makebox(0,0)[t]{$\zeta$}}
\put(-30,50){\makebox(0,0)[r]{$(c)$}}
\end{picture}
\hspace*{40pt}
\begin{picture}(140,90)
\put(0,0){\vector(0,1){90}} \put(0,0){\vector(1,0){140}}
{\thicklines
\put(0,0){\line(1,2){20}}
\qbezier(20,40)(30,70)(140,90)
}
\put(-4,88){\makebox(0,0)[r]{$\psi(\alpha)$}}
\put(138,-4){\makebox(0,0)[t]{$\alpha$}}
\end{picture}
\medskip
\caption{Graphs of $\phi$ and $\psi$: $(a)$ $\partial\phi(\zeta_{lim})=\emptyset$, $(b)$ $\partial\phi(\zeta_{lim})\not=\emptyset$, $(c)$ $\zeta_{lim}=+\infty$.}
\label{fig_schemes}
\end{figure}
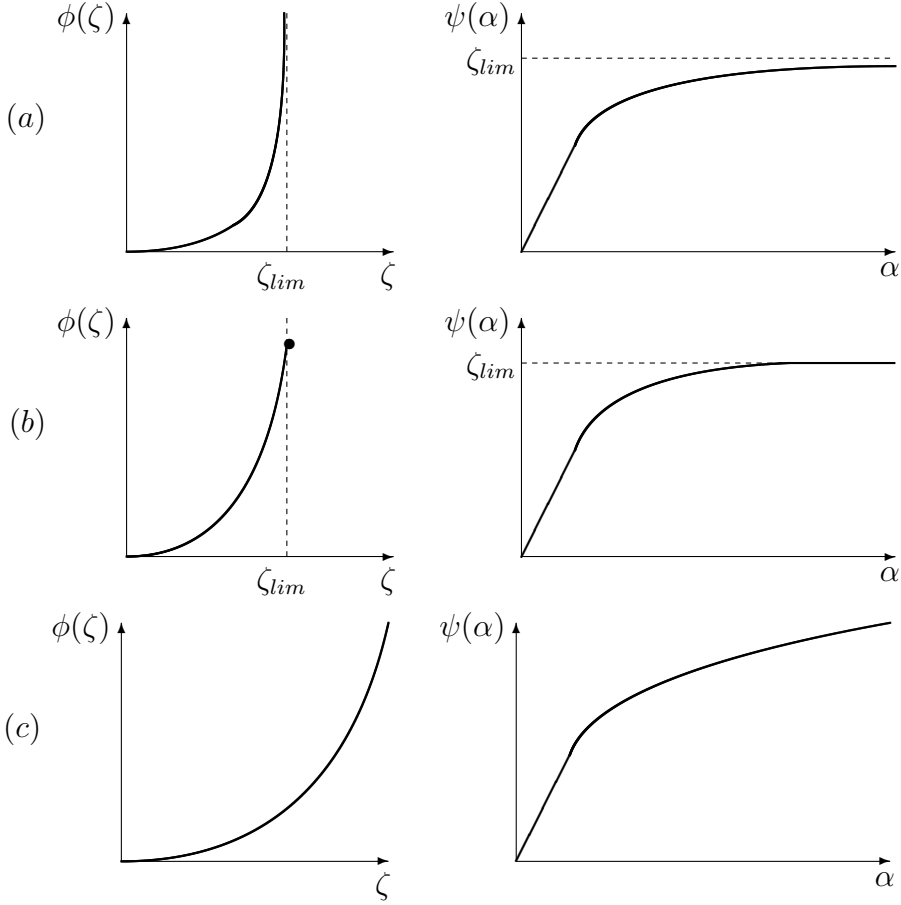

\section{Stress and displacement problems for given $\alpha\in\mathbb R_+$}
\label{sec_problems_alpha}

In this section, we formulate new variational problems  in terms of stresses and displacements enabling us to compute function values $\psi(\alpha)$ for $\alpha\in\mathbb R_+$. The parameter $\alpha$ will be used to control the loading process and to get the respective loading path $\mbox{graph}[\psi]$ for a larger class of yield functions than in \cite{SHHC13}.

To derive the formulation in terms of stresses, we introduce the following set:
$$\tilde\Lambda_L=\bigcup_{\tilde\zeta\in\mathbb R_+}\Lambda_{\tilde\zeta L}=\{\tilde\zeta\tau\ |\; \tilde\zeta\in\mathbb R_+,\;\tau\in\Lambda_L\}.$$
Clearly, $\tilde\Lambda_L$ is a closed, convex and non-empty subset of $S$ and for any $\tau\in\tilde\Lambda_L$ there exists a unique loading parameter $\tilde\zeta$ such that $\tau\in\Lambda_{\tilde\zeta L}$ owing to $(L2)$. To stress that $\tau\in\Lambda_{\tilde\zeta L}$ with $\tilde\zeta\in\mathbb R_+$ and using the fact that such $\tilde\zeta$ is unique, we shall write $\tilde\zeta=\omega(\tau)$ in what follows. It is readily seen that the function 
$$\omega:\tilde\Lambda_{L}\rightarrow\mathbb R_+$$
is concave in $\tilde\Lambda_{L}$ and satisfies the relation
$$\omega(\lambda\tau_1+(1-\lambda)\tau_2)=\lambda\omega(\tau_1)+(1-\lambda)\omega(\tau_2) \quad\forall \tau_1,\tau_2\in\tilde\Lambda_{ L},\;\forall\lambda\in[0,1].$$
Moreover,
\begin{equation}
\bigcup_{\tilde\zeta\in\mathcal D}\Lambda_{\tilde\zeta L}\cap P=\tilde\Lambda_{L}\cap P.
\label{tilde_lambda_cap_P}
\end{equation}

Let $\alpha>0$ be given and $\zeta=\psi(\alpha)$. Then,
\begin{eqnarray*}
\phi(\zeta)-\alpha\zeta&\stackrel{(\ref{zeta_min})}{=}&\inf_{\tilde\zeta\geq0}\{\phi(\tilde\zeta)-\alpha\tilde\zeta\}=\inf_{\tilde\zeta\in\mathcal D}\left\{\inf_{\tau\in{\Lambda_{\tilde\zeta L}}\cap P}\mathcal{I}(\tau)-\alpha\tilde\zeta\right\}=\\
&=&\inf_{\tilde\zeta\in\mathcal D}\inf_{\tau\in{\Lambda_{\tilde\zeta L}}\cap P}\left\{\mathcal{I}(\tau)-\alpha\omega(\tau)\right\}=\inf_{\tau\in{\tilde\Lambda_{L}}\cap P}\left\{\mathcal{I}(\tau)-\alpha\omega(\tau)\right\}
\end{eqnarray*}
using the definition of $\omega$ and (\ref{tilde_lambda_cap_P}).

On basis of this result we formulate the following problem in terms of stresses: given $\alpha\geq0$,
$$ (\mathcal P^*)^{\alpha}\qquad\mbox{find }\sigma:=\sigma(\alpha)\in\tilde\Lambda_L\cap P: \quad \mathcal{I}(\sigma)-\alpha\omega(\sigma) \leq \mathcal{I}(\tau)-\alpha\omega(\tau)\quad \forall\tau \in \tilde\Lambda_{L} \cap P.$$
Properties of the functions $\mathcal{I}$ and $\omega$ ensure that for any $\alpha\geq0$ problem $ (\mathcal P^*)^{\alpha}$ has a unique solution $\sigma$. Moreover, $\zeta=\psi(\alpha)=\omega(\sigma)$ and $\sigma$ also solves $(\mathcal P^*)_{\zeta}$. Conversely, if $\sigma$ is the unique solution to $(\mathcal P^*)_{\zeta}$, $\zeta\in\mathcal D\setminus\{0\}$,  then $\sigma$ also solves $ (\mathcal P^*)^{\alpha}$ for $\alpha\in\partial \phi(\zeta)$.

Now, we derive the dual problem to  $ (\mathcal P^*)^{\alpha}$ in terms of displacements for given $\alpha>0$. Let $\zeta=\psi(\alpha)>0$. Then,
\begin{eqnarray*}
\phi(\zeta)-\alpha\zeta&\stackrel{(\ref{zeta_min})}{=}&\inf_{\tilde\zeta\geq0}\{\phi(\tilde\zeta)-\alpha\tilde\zeta\}=\inf_{\tilde\zeta\geq0}\left\{\inf_{\tau\in{\Lambda_{\tilde\zeta L}}\cap P}\mathcal{I}(\tau)-\alpha\tilde\zeta\right\}=\\
&\stackrel{(\ref{duality})}{=}&\inf_{\tilde\zeta\geq0}\left\{\sup_{v\in\mathbb V}[-J_{\tilde\zeta}(v)]-\alpha\tilde\zeta\right\}=\\
&=&\inf_{\tilde\zeta\geq0}\sup_{v\in\mathbb V}\mathcal L(\tilde\zeta,v),
\end{eqnarray*}
where
$$\mathcal L(\tilde\zeta,v)=-\Psi(\varepsilon(v))+\tilde\zeta(L(v)-\alpha), \quad (\tilde\zeta,v)\in\mathbb R_+\times\mathbb V$$
and $\Psi$ is defined by (\ref{Psi}). From \cite[Proposition VI.2.3]{ET74}, it follows that
\begin{equation}
\phi(\zeta)-\alpha\zeta=\inf_{\tilde\zeta\geq0}\sup_{v\in\mathbb V}\mathcal L(\tilde\zeta,v)=\sup_{v\in\mathbb V}\inf_{\tilde\zeta\geq0}\mathcal L(\tilde\zeta,v)=-\inf_{v\in\mathbb V,\;L(v)\geq\alpha}\Psi(\varepsilon(v)).
\label{ET_label}
\end{equation}
Since $\Psi$ is convex on $S$ and $\Psi(0)=0$ it holds:
\begin{equation}
\inf_{v\in\mathbb V,\;L(v)\geq\alpha}\Psi(\varepsilon(v))=\inf_{v\in\mathbb V^\alpha}\Psi(\varepsilon(v)),
\label{inf_alpha}
\end{equation}
where
$$\mathbb V^\alpha=\{v\in\mathbb V\ |\;\; L(v)=\alpha\}.$$
Indeed, for any $v\in\mathbb V$, $L(v)>\alpha$, one can set $w=\frac{\alpha}{L(v)}v$ belonging to $\mathbb V^\alpha$ and satisfying
$$\Psi(\varepsilon(w))\leq\frac{\alpha}{L(v)}\Psi(\varepsilon(v))\leq\Psi(\varepsilon(v)).$$
Therefore, the problem in terms of displacements for given $\alpha>0$ reads as follows:
$$ (\mathcal P)^{\alpha}\qquad\mbox{find } u:=u(\alpha)\in\mathbb V^\alpha:\qquad\Psi(\varepsilon(u))\leq\Psi(\varepsilon(v))\qquad\forall v\in\mathbb V^\alpha.$$
This and (\ref{ET_label}) yield
$$\inf_{v\in\mathbb V^\alpha}\Psi(\varepsilon(v))=-\inf_{\tau \in \tilde\Lambda_{L} \cap P}[\mathcal{I}(\tau)-\alpha\omega(\tau)],$$
i.e., $(\mathcal P)^\alpha$ and $(\mathcal P^*)^\alpha$ are mutually dual. Notice that this result can also be derived using some parts of the proof of Lemma 5.2 in \cite{T85}.  Solvability of $(\mathcal P)^\alpha$ is problematic on $\mathbb V$ from the same reasons as in the case of $(\mathcal P)_\zeta$. However, this formulation is useful for numerical realization of its discretization. If we admit that $(\mathcal P)^\alpha$ has a solution for some $\alpha>0$ then the following result holds.

\begin{theorem}
Suppose that there exists a solution $u$ to $(\mathcal P)^\alpha$, $\alpha>0$. Then 
\begin{equation}
\zeta=\psi(\alpha)=\frac{1}{\alpha}\langle\Sigma(\varepsilon(u)), \varepsilon(u)\rangle.
\label{zeta_from_alpha}
\end{equation}
In addition, $u$ is the solution to $(\mathcal P)_\zeta$ and $\sigma=\Sigma(\varepsilon(u))$ is the solution to problems $(\mathcal P^*)^\alpha$ and $(\mathcal P^*)_\zeta$.

Conversely, if $u$ is a solution to $(\mathcal P)_\zeta$ then $u$ also solves $(\mathcal P)^\alpha$ for $\alpha=L(u)$.
\label{th_primal_alpha}
\end{theorem}

\begin{proof}
Let $u$ be a solution to $(\mathcal P)^\alpha$, $\alpha>0$ and $\zeta=\psi(\alpha)>0$. Then using (\ref{ET_label}), (\ref{inf_alpha}), the pair $(\zeta,u)$ is a saddle point of the Lagrangian $\mathcal L$:
$$\mathcal L(\zeta,v)\leq\mathcal L(\zeta,u)\leq\mathcal L(\tilde\zeta,u)\quad\forall (\tilde\zeta,v)\in\mathbb R_+\times\mathbb V,$$ 
or equivalently
\begin{equation}
\left\{
\begin{array}{l}
L(u)=\alpha,\;\;\zeta>0,\\
\langle\Sigma(\varepsilon(u)),\varepsilon(v)\rangle=\zeta L(v)\quad\forall v\in \mathbb V,
\end{array}
\right.
\label{saddle_equiv}
\end{equation}
i.e. $u$ solves $(\mathcal P)_\zeta$. Consequently, $\sigma=\Sigma(\varepsilon(u))$ solves $(\mathcal P^*)_\zeta$ and also $(\mathcal P^*)^\alpha$. Moreover, inserting $v=u$ into (\ref{saddle_equiv})$_2$, we obtain (\ref{zeta_from_alpha}).

Conversely, let $u$ be a solution to $(\mathcal P)_\zeta$ for $\zeta\in\mathcal D$ and denote $\alpha:=L(u)$. Then $u\in\mathbb V^\alpha$ and
$$\Psi(\varepsilon(u))=J_\zeta(u)+\zeta\alpha\leq\inf_{v\in V^\alpha}J_\zeta(v)+\zeta\alpha=\inf_{v\in V^\alpha}\Psi(\varepsilon(v)).$$
Hence, $u$ is the solution to $(\mathcal P)^\alpha$.
\end{proof}

\begin{remark}
\emph{Theorem \ref{th_primal_alpha} expresses the relation between $\zeta$ and $\alpha$ through displacements. If $u$ is a solution to $(\mathcal P)_\zeta$ then $\alpha=L(u)$. Therefore, one can say that $\alpha$ represents work of external forces. The equality $\alpha=L(u)$ is in accordance with \cite{CHKS14, SHHC13}.}
\end{remark}

\section{Discretization and convergence analysis}
\label{sec_discretization}

\subsection{Setting of discretized problems}

For the sake of simplicity,  we now suppose that $\Omega$ is a {\it polyhedral} domain. Let $\left\{ \mathcal{T}_h \right\},\ h> 0$ be a collection of regular partitions of $\overline\Omega$ into tetrahedrons $\triangle$ which are consistent with the decomposition of $\partial\Omega$ into $\Gamma_D$ and $\Gamma_N$. Here, $h$ is a positive mesh size parameter. With any $\mathcal{T}_h$ we associate the following finite-dimensional spaces:
$$
  \mathbb{V}_h = \{ v_h \in C(\overline{\Omega};\mathbb R^3)\;\; |\quad v_h|_\triangle \in P_1(\triangle;\mathbb R^3)\;\; \forall \triangle \in \mathcal{T}_h,\quad v_h=0\ \mbox{on } \Gamma_D \}{,}
$$
$$
  S_h = \{ \tau_h \in S\quad |\quad \tau_h|_\triangle \in P_0(\triangle;\mathbb R^{3\times 3}_{sym})\quad \forall \triangle\in \mathcal{T}_h \},
$$
where $P_k(\triangle),\ k\geq 0$ integer, stands for the space of all polynomials of degree less or equal $k$ defined in $\triangle \in \mathcal{T}_h$. The spaces $\mathbb{V}_h$ and $S_h$ are the simplest finite element approximations of $\mathbb{V}$ and $S$, respectively. Next we shall suppose that 
$\overline{\mathbb V\cap C^\infty(\overline\Omega;\mathbb R^3)}=\mathbb V.$
 Further, define the following convex sets:
$$
 P_h=P\cap S_h,
$$
$$
  \Lambda^h_{\zeta L} = \left\{\tau_h\in S_h\quad |\quad \langle \tau_h, \varepsilon(v_h)\rangle =\zeta L(v_h)\quad \forall v_h\in \mathbb V_h \right\},\quad\zeta\geq0,
$$
$$
  \tilde\Lambda^h_{L} =\bigcup_{\zeta\in\mathbb R_+}\Lambda_{\zeta L}^h,
$$
$$\mathbb V_h^\alpha=\{v_h\in\mathbb V_h\;|\;\; L(v_h)=\alpha\},\quad\alpha\geq0,$$
$$\mathcal D_h:=\{\zeta\in\mathbb R_+\ |\;  \Lambda_{\zeta L}^h\cap P_h \not= \emptyset\},$$
which are natural discretizations of $P$, $\Lambda_{\zeta L}$, $\tilde\Lambda_L$, $\mathbb V^\alpha$, and $\mathcal D$, respectively. 
We also consider the functions $\phi_h$, $\psi_h$, $\omega_h$ and the limit load parameter $\zeta_{lim,h}$ with the analogous definitions and properties as their continuous counterparts.

The discrete versions of $(\mathcal{P}^*)_\zeta$,  $(\mathcal{P}^*)^\alpha$, $(\mathcal{P})_\zeta$,  $(\mathcal{P})^\alpha$ for given $\zeta\geq0$ or $\alpha\geq0$ read as follows: 
\begin{equation*}
(\mathcal{P}^{*}_h)_\zeta\qquad
   \mbox{find }\sigma_h:=\sigma_h(\zeta)\in\Lambda_{\zeta L}^h \cap P_h: \quad \mathcal{I}(\sigma_h) \leq \mathcal{I}(\tau_h) ,\quad \forall\tau_h \in \Lambda_{\zeta L}^h \cap P_h,
\end{equation*}
$$ (\mathcal P^*_h)^{\alpha}\qquad\mbox{find }\sigma_h:=\sigma_h(\alpha)\in\tilde\Lambda_L^h\cap P_h: \quad \mathcal{I}(\sigma_h)-\alpha\omega_h(\sigma_h) \leq \mathcal{I}(\tau_h)-\alpha\omega_h(\tau_h) ,\quad \forall\tau_h \in \tilde\Lambda_{L}^h \cap P_h,$$
\begin{equation*}
(\mathcal{P}_h)_\zeta\qquad
   \mbox{find } u_h:=u_h(\zeta)\in\mathbb V_h: \quad J_\zeta(u_h)\leq J_\zeta(v_h),\quad\forall v_h\in \mathbb V_h,
\end{equation*}
$$ (\mathcal P_h)^{\alpha}\qquad\mbox{find } u_h:=u_h(\alpha)\in\mathbb V_h^\alpha:\qquad\Psi(\varepsilon(u_h))\leq\Psi(\varepsilon(v_h))\qquad\forall v_h\in\mathbb V_h^\alpha.$$
Clearly problems $(\mathcal{P}_h^*)_\zeta$ and $ (\mathcal P^*_h)^{\alpha}$ have unique solutions for any $\zeta\in\mathcal D_h$, $\alpha\geq0$ and $h>0$. Further, the existence of solutions to $(\mathcal{P}_h)_\zeta$ and $ (\mathcal P_h)^{\alpha}$ is guaranteed for any $\zeta\in[0,\zeta_{lim,h})$, $\alpha\geq0$ and $h>0$, see e.g. \cite{FG83, SHHC13}. The mutual relations among the solutions to these problems remain the same as in the continuous setting. The relation between $\zeta$  and $\alpha$ is defined using the functions $\phi_h$ and $\psi_h$, analogously to the continuous case: $\alpha\in\partial\phi_h(\zeta)$ if $\zeta\in\mathcal D_h\setminus\{0\}$, $\alpha=0$ if $\zeta=0$ and $\zeta=\psi_h(\alpha)$. In particular,
\begin{equation}
\zeta=\psi_h(\alpha)=\frac{1}{\alpha}\langle\Sigma(\varepsilon(u_h)), \varepsilon(u_h)\rangle,
\label{zeta_from_alpha_h}
\end{equation}
where $u_h$ is any solution to $(\mathcal{P}_h)^\alpha$. It is worth noticing that (\ref{zeta_from_alpha_h}) enables us to express $\zeta$ elementwise: $$\zeta=\sum_{\triangle\in\mathcal T_h}\zeta_\triangle,\quad \zeta_\triangle=\frac{|\triangle|}{\alpha}\Sigma(\varepsilon(u_h)|_\triangle): \varepsilon(u_h)|_\triangle.$$

\subsection{Convergence analysis}

In what follows, we study convergence of $(\mathcal{P}_h^*)_\zeta$, $(\mathcal P^*_h)^{\alpha}$ and $\psi_h$ to their continuous counterparts when the discretization parameter $h\rightarrow0_+$. To this end we need the following well-known results \cite{M77, HHl82}.

\begin{lemma}
For any $v\in\mathbb V$ there exists a sequence $\{v_h\}$, $v_h\in\mathbb V_h$ such that $v_h\rightarrow v$ in $\mathbb V$ as $h\rightarrow+\infty$.
\label{lem_V_h_conv}
\end{lemma}

\begin{lemma}
Let  $r_h:\ S \rightarrow S_h$ be the orthogonal projection of $S$ on $S_h$ with respect to the scalar product $\langle\cdot,\cdot\rangle$, i.e.,
$$
 r_h\tau|_\triangle = \dfrac{1}{|\triangle|} \int_\triangle\tau\, {\mbox{d}}x\qquad \forall \triangle\in \mathcal{T}_h \quad \forall \tau\in S.
$$
Then $ r_h\tau\in P_h$ for any $\tau\in  P$,  $r_h\tau\in \Lambda_{\zeta L}^h$ for any $\tau\in  \Lambda_{\zeta L},\;\zeta\geq0$ and
\begin{equation*} \label{equ.3.1}
  r_h\tau \rightarrow \tau\quad \mbox{in } S {\mbox{ as }} h \rightarrow 0{+}.
\end{equation*}
\label{lem_r_h}
\end{lemma}

\begin{corollary}
$\zeta_{lim,h}\geq\zeta_{lim}$ for any $h>0.$
\end{corollary}

\begin{proof}
It is sufficient to show that $\mathcal D\subset\mathcal D_h$ for any $h>0.$ If $\zeta\in\mathcal D$ then there exists $\tau\in \Lambda_{\zeta L}\cap P$. From Lemma \ref{lem_r_h}, $r_h\tau\in\Lambda_{\zeta L}^h\cap P_h$ for any $h>0$. Therefore, $\zeta\in\mathcal D_h$  for any $h>0.$
\end{proof}

\begin{lemma}
Let $\tau\in S$ and $\{\tau_h\}$, $\tau_h\in S_h$ be a sequence such that $\tau_h\in \Lambda^h_{\zeta L}\cap P_h$, $\zeta\geq0$ and $\tau_h \rightharpoonup \tau$  (weakly) in  $S$ as $h\rightarrow0_+$. Then $\tau\in \Lambda_{\zeta L}\cap P$.
\label{lem_density}
\end{lemma}

The following convergence result is a direct consequence of Lemmas \ref{lem_V_h_conv}--\ref{lem_density}.
\begin{theorem}
Let $\zeta\in\mathcal D$ and $\sigma_h$ be a solution to $(\mathcal{P}^*_h)_\zeta$, $h\rightarrow0_+$. Then 
$$\sigma_h\rightarrow\sigma\;\;\mbox{in }S,\;\;h\rightarrow0_+,$$ 
$$\phi_h(\zeta)\rightarrow\phi(\zeta),\;\;h\rightarrow0_+,$$
where $\sigma\in \Lambda_{\zeta L}\cap P$ is the unique solution to $(\mathcal{P}^*)_\zeta$.
\label{thm_zeta_conv}
\end{theorem}

To prove convergence of solutions of $(\mathcal{P}^*_h)^\alpha$ to a solution of $(\mathcal{P}^*)^\alpha$, we need some other auxilliary results.
\begin{lemma}
For any $v\in\mathbb V^{\alpha=1}$, there exists a sequence $\{w_h\}$, $w_h\in\mathbb V_h^{\alpha=1}$ such that $w_h\rightarrow v$ in $\mathbb V$ as $h\rightarrow+\infty$.
\label{lem_V_h^1_conv}
\end{lemma}

\begin{proof}
Let $v\in\mathbb V^{\alpha=1}$ and $\{v_h\}$, $v_h\in\mathbb V_h$ be a sequence such that $v_h\rightarrow v$ in $\mathbb V$ as $h\rightarrow+\infty$. Then, $L(v_h)\rightarrow L(v)=1$ as $h\rightarrow0_+$ and $w_h=\frac{1}{L(v_h)}v_h\in\mathbb V_h^{\alpha=1}$ has the required property.
\end{proof}

\begin{lemma}
There exists a constant $c>0$ such that for any sufficiently small $h>0$
$$\omega_h(\tau_h)\leq c\|\tau_h\|_{C^{-1}}\quad\forall \tau_h\in\tilde\Lambda_L^h.$$
\label{omega_bound}
\end{lemma}

\begin{proof}
Let $v\in\mathbb V^{\alpha=1}$ and $\epsilon>0$ be given. Then, there exists a sequence $\{w_h\}$, $w_h\in\mathbb V_h^{\alpha=1}$ such that $w_h\rightarrow v$ in $\mathbb V$ as $h\rightarrow+\infty$. Hence, 
$$\exists h_0>0:\quad \|\varepsilon(w_h)\|_C\leq\|\varepsilon(v)\|_C+\epsilon\quad \forall 0<h\leq h_0$$
and using the definition of $\tilde\Lambda^h_{L}$, we obtain
$$\omega_h(\tau_h)=\omega_h(\tau_h)L(w_h)=\langle\tau_h,\varepsilon(w_h)\rangle\leq\|\tau_h\|_{C^{-1}}\|\varepsilon(w_h)\|_C\leq c\|\tau_h\|_{C^{-1}}\quad\forall 0<h\leq h_0,\;\forall \tau_h\in\tilde\Lambda_L^h,$$
where $c=\|\varepsilon(v)\|_C+\epsilon$.
\end{proof}

\begin{lemma}
Let $\{\tau_h\}$, $\tau_h\in \tilde\Lambda^h_{L}\cap P_h$ be such that $\tau_h \rightharpoonup \tau$  (weakly) in $S$ and $\omega_h(\tau_h)\rightarrow\zeta$ as $h\rightarrow0_+$. Then $\tau\in \tilde\Lambda_{L}\cap P$ and $\omega(\tau)=\zeta$. 
\label{lem_density2}
\end{lemma}
\begin{proof}
Since $\tau_h \rightharpoonup \tau$ and $P$ is a closed convex set, $\tau\in P$. Let $v\in \mathbb V$ and $\{v_h\}$, $v_h\in\mathbb V_h$ be such that $v_h\rightarrow v$ in $\mathbb V$ as $h\rightarrow+\infty$.
From the definition of $\tilde\Lambda^h_{L}$, it follows that
$$\langle\tau_h,\varepsilon(v_h)\rangle=\omega_h(\tau_h)L(v_h).$$
Passing to the limit with $h\rightarrow0_+$, we conclude that $\tau\in \tilde\Lambda_{L}\cap P$ and $\omega(\tau)=\zeta$.
\end{proof}

\begin{theorem}
Let $\alpha\geq0$ be given and $\{\sigma_h\}$ be a sequence of solutions to $(\mathcal{P}^*_h)^\alpha$, $h>0$. Then $\sigma_h\rightarrow\sigma$ in $S$, $\omega_h(\sigma_h)\rightarrow\omega(\sigma)$ and $\psi_h(\alpha)\rightarrow\psi(\alpha)$ as $h\rightarrow0_+$, where $\sigma$ is a solution to $(\mathcal{P}^*)^\alpha$.
\label{thm_alpha_conv}
\end{theorem}

\begin{proof}
The proof consists of three steps.

{\it Step 1 (Boundedness)}. Let $\tau\in\tilde\Lambda_L\cap P$ be fixed. Then $r_h\tau\in\Lambda_{\omega(\tau) L}^h\cap P_h\subset\tilde\Lambda_L^h\cap P_h$ and $r_h\tau\rightarrow \tau$ in $S$ as $h\rightarrow0_+$. From the definition of $(\mathcal{P}^*_h)^\alpha$ it follows:
$$\mathcal{I}(\sigma_h)-\alpha\omega_h(\sigma_h)\leq \mathcal{I}(r_h\tau)-\alpha\omega(\tau)\quad\forall h>0$$
since $\omega_h(r_h\tau)=\omega(\tau)$. From this and Lemma \ref{omega_bound}, we obtain
$$\exists c_1>0, c_2\in\mathbb R, h_0>0:\quad \frac{1}{2}\|\sigma_h\|^2_{C^{-1}}=\mathcal{I}(\sigma_h)\leq c_1\|\sigma_h\|_{C^{-1}}+c_2\quad\forall h\in(0,h_0).$$
This implies boundedness of $\{\sigma_h\}$ and consequently boundedness of $\{\omega_h(\sigma_h)\}$.

{\it Step 2 (Weak convergence)}. One can pass to subsequences $\{\sigma_{h'}\}\subset\{\sigma_h\}$ and $\{\omega_{h'}(\sigma_{h'})\}\subset\{\omega_{h}(\sigma_{h})\}$ such that
\begin{equation}
\left.
\begin{array}{c c l}
\sigma_{h'}\rightharpoonup\sigma\;\;\mbox{in }S & as & h'\rightarrow0_+,\\
\omega_{h'}(\sigma_{h'})\rightarrow\zeta& as & h'\rightarrow0_+.\\
\end{array}
\right\}
\label{weak_label}
\end{equation}
From Lemma \ref{lem_density2}, it follows that $\sigma\in\tilde\Lambda_L\cap P$ and $\zeta=\omega(\sigma)$. Let $\tau\in\tilde\Lambda_L\cap P$ be arbitrary. Then $r_{h'}\tau\in\Lambda_{\omega(\tau) L}^{h'}\cap P_{h'}\subset\tilde\Lambda_L^{h'}\cap P_{h'}$, $\omega_{h'}(r_{h'}\tau)=\omega(\tau)$ and $r_{h'}\tau\rightarrow \tau$ in $S$ as $h'\rightarrow0_+$. Hence,
$$\mathcal{I}(\sigma)-\alpha\omega(\sigma)\leq\liminf_{h'\rightarrow0_+}[\mathcal{I}(\sigma_{h'})-\alpha\omega_{h'}(\sigma_{h'})]\leq\liminf_{h'\rightarrow0_+}[\mathcal{I}(r_{h'}\tau)-\alpha\omega(r_{h'}\tau)]=\mathcal{I}(\tau)-\alpha\omega(\tau),$$
i.e., $\sigma$ is the solution to $(\mathcal P^*)^\alpha$. Since $(\mathcal P^*)^\alpha$ has a unique solution, (\ref{weak_label}) holds for the whole sequence. Consequently,
$$\psi_h(\alpha)=\omega_h(\sigma_h)\rightarrow\omega(\sigma)=\psi(\alpha) \quad\mbox{as}\quad h\rightarrow0_+.$$

{\it Step 3  (Strong convergence).} Since $r_h\sigma\in\tilde\Lambda_L^h\cap P_h$, $\omega_h(r_h\sigma)=\omega(\sigma)$ and $r_h\sigma\rightarrow\sigma$ in $S$ as $h\rightarrow0_+$, we have
\begin{eqnarray*}
\mathcal{I}(\sigma)&\leq&\liminf_{h\rightarrow0_+}\mathcal{I}(\sigma_h)\leq\limsup_{h\rightarrow0_+}\mathcal{I}(\sigma_h)=\limsup_{h\rightarrow0_+}[\mathcal{I}(\sigma_h)-\alpha\omega_h(\sigma_h)]+\alpha\omega(\sigma)\\
&\leq&\lim_{h\rightarrow0_+}[\mathcal{I}(r_h\sigma)-\alpha\omega_{h}(r_h\sigma)]+\alpha\omega(\sigma)=\mathcal{I}(\sigma).
\end{eqnarray*}
Therefore,
$$\|\sigma_h\|^2_{C^{-1}}=2\mathcal{I}(\sigma_h)\rightarrow 2\mathcal{I}(\sigma)=\|\sigma\|^2_{C^{-1}}\quad\mbox{as}\quad h\rightarrow0_+,$$
which implies strong convergence of $\{\sigma_h\}$ to $\sigma$ in S.
\end{proof}

\begin{remark}
\emph{We summarize the properties of the functions $\psi$ and $\psi_h$, $h>0$: 
\begin{itemize}
\item[$a)$] $\psi$ and $\psi_h$ are nondecreasing and continuous in $\mathbb R_+$ for any $h>0$;
\item[$b)$] $\psi(\alpha)\rightarrow\zeta_{lim}$,  $\psi_h(\alpha)\rightarrow\zeta_{lim,h}$ as $\alpha\rightarrow+\infty$, for any $h>0$;
\item[$c)$] $\zeta_{lim,h}\geq\zeta_{lim}\geq\psi(\alpha)$ for any $h>0$ and $\alpha\geq0$;
\item[$d)$] $\psi_h(\alpha)\rightarrow\psi(\alpha)$ as $h\rightarrow0_+$ for any $\alpha\geq0$.
\end{itemize}
}
\label{remark_conv}
\end{remark}
Notice that from Remark \ref{remark_conv} $b),d)$ it follows that for any $\epsilon>0$ there exists $\alpha$ large enough and $h_0>0$ small enough such that $|\psi_h(\alpha)-\zeta_{lim}|<\epsilon\;\;\forall h\leq h_0$. Direct convergence of $\zeta_{lim,h}$ to $\zeta_{lim}$ is guaranteed only for some yield functions $\Phi$ as follows from the next theorem.

\begin{theorem}
Let the yield function $\Phi$ be coercive on $\mathbb{R}_{sym}^{3\times3}$ and the assumptions $(L1), (L2)$ be satisfied. Then
\begin{equation}
\zeta_{lim,h}\rightarrow\zeta_{lim}\quad\mbox{as}\quad h\rightarrow0_+.
\end{equation}
\label{th_lim_conv}
\end{theorem}

\begin{proof} 
Coerciveness of $\Phi$ ensures that the set $P$ is bounded in $L^\infty(\Omega;\mathbb R^{3\times 3}_{sym})$, i.e.
\begin{equation}
\exists c>0:\quad |\tau_{ij}(x)|\leq c\quad\forall \tau\in P,\;\forall i,j=1,2,3,\;\mbox{for a.a. }x\in\Omega.
\label{P_bound}
\end{equation}
Next, we show that $\{\zeta_{lim,h}\}$ is bounded. Consider a bounded sequence $\{w_h\}$, $w_h\in\mathbb V_h^{\alpha=1}$:
\begin{equation}
\exists M>0:\quad\|\varepsilon(w_h)\|_{L^1(\Omega;\mathbb R^{3\times 3}_{sym})}\leq M \quad\forall h>0.
\label{M_bound}
\end{equation}
The existence of such a sequence is guaranteed by Lemma  \ref{lem_V_h^1_conv}. Then for any $\zeta\in\mathcal D_h$ and $\tau_h\in\Lambda^h_{\zeta L}\cap P_h$ it holds
$$\zeta=\zeta L(w_h)=\langle\tau_h,\varepsilon(w_h)\rangle\stackrel{(\ref{P_bound})}{\leq}c\|\varepsilon(w_h)\|_{L^1(\Omega;\mathbb R^{3\times 3}_{sym})}\stackrel{(\ref{M_bound})}{\leq}cM.$$
Hence, $\zeta_{lim,h}\leq cM<+\infty$  for any $h>0$. In addition, from boundedness of $P$, it follows that $\zeta_{lim,h}\in\mathcal D_h$ for any $h>0$.

Let $\{\tau_h\}$, $h>0$, be such that $\tau_h\in\Lambda_{\zeta_{lim,h} L}^h\cap P_h$. Then $\{\tau_h\}$ is bounded in $S$ and there exist subsequences $\{\zeta_{lim,h'}\}$ and $\{\tau_{h'}\}$, $\tau_{h'}\in\Lambda_{\zeta_{lim,h'} L}^{h'}\cap P_{h'}$ such that
$$\tau_{h'}\rightharpoonup\tau\;\;\mbox{in }S,\quad \zeta_{lim,h'}\rightarrow\hat\zeta,\quad h'\rightarrow0_+.$$
Clearly, $\tau\in\Lambda_{\hat\zeta L}\cap P$ and thus $\hat\zeta\in\mathcal D$. Therefore $\hat\zeta=\zeta_{lim}$ using Corollary 5.1.
\end{proof}

\section{Numerical experiments}
\label{sec_eval}

In order to verify the previous theoretical results, we have performed several numerical experiments with two yield functions presented below. Problem $(\mathcal P_h)^\alpha$ which is needed for the evaluation of $\psi_h(\alpha)$ is solved by a regularized semismooth Newton method. This method has been proposed and theoretically justified in \cite[ALG3]{CHKS14}. Each iterative step leads to a quadratic programming problem. After finding a solution $u_h:=u_h(\alpha)$ of  $(\mathcal P_h)^\alpha$, the value $\zeta=\psi_h(\alpha)$ of the load parameter is computed by (\ref{zeta_from_alpha_h}).

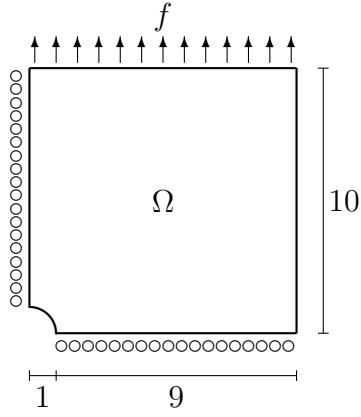
\begin{figure}[htbp]
        \centering
        \begin{picture}(140,140)

            {\thicklines
            \put(30,20){\line(1,0){90}}
            \put(20,120){\line(1,0){100}}
            \put(20,30){\line(0,1){90}}
            \put(120,20){\line(0,1){100}}
            \put(20,20){\oval(20,20)[rt]}
            }
            \multiput(15,32)(0,5){18}{\circle{4}}
            \multiput(32,15)(5,0){18}{\circle{4}}
            \multiput(22,122)(8,0){13}{\vector(0,1){10}}

            \put(70,134){\makebox(0,0)[b]{$f$}}
            \put(70,70){\makebox(0,0)[c]{$\Omega$}}
            \put(25,0){\makebox(0,0)[t]{1}}
            \put(75,0){\makebox(0,0)[t]{9}}
            \put(132,70){\makebox(0,0)[l]{10}}
            \put(20,4){\line(1,0){100}}
            \put(130,20){\line(0,1){100}}
            \put(20,2){\line(0,1){4}}
            \put(30,2){\line(0,1){4}}
            \put(120,2){\line(0,1){4}}
            \put(128,20){\line(1,0){4}}
            \put(128,120){\line(1,0){4}}

        \end{picture}
\caption{Geometry of the plane strain problem.}
\label{fig_scheme}
\end{figure}

The performed experiments are related to a plain strain problem with $\Omega$ depicted in Figure \ref{fig_scheme}: $\Omega$ is a quarter of the square containing the circular hole of radius 1 in its center. The constant traction of density $f=(0,450),(0,0)$ is applied on the upper, and the right vertical side, respectively. This load corresponds to $\zeta=1$. On the rest of $\partial\Omega$ the symmetry boundary conditions are prescribed. We consider  linear Hooke's law for a homogeneous, isotropic elastic material:
\begin{equation} 
  \tau = Ce\ \Leftrightarrow\ \tau = {\lambda\, \mbox{tr}(e)\, \iota} + 2\mu e,\quad e,\tau\in \mathbb{R}^{3\times3}_{sym},
\label{hooke_law}
\end{equation}
where $\iota$ is the $(3\times 3)$ identity matrix, $\mbox{tr}(e) = e_{ii}$ is the trace of $e$ and $\lambda=\frac{E\nu}{(1+\nu)(1-2\nu)}$, $\mu=\frac{E}{2(1+\nu)}$ are positive constants representing Lame's coefficients. 
The elastic material parameters are set as follows: $E = 206 900$ (Young's modulus) and $\nu = 0.29$ (Poisson ratio). 

The loading paths represented by the graph of $\psi_h:\alpha\mapsto\zeta$ are compared for seven different meshes with 1080, 2072, 3925, 10541, 23124, 41580 and 92120 nodes. The problem is implemented in MatLab.

\subsection{Yield function 1}

Consider the yield function
$$\Phi(\tau)=\sqrt{C^{-1}\tau:\tau}, \quad \tau\in\mathbb{R}^{3\times3}_{sym}$$
(a similar yield function has been considered in, e.g., \cite{R09, CG10}).
Then
$$\left(\Sigma(e)\right)(x)=\mathbb D \Psi(e)(x)=\left\{
\begin{array}{r l}
Ce(x), & \sqrt{Ce(x):e(x)}\leq\gamma,\\[3pt]
\frac{\gamma}{\sqrt{Ce(x):e(x)}}Ce(x), & \sqrt{Ce(x):e(x)}\geq\gamma,
\end{array}
\right.,\quad \forall e\in S,\;\;\mbox{for a.a. }x\in \Omega,$$
$$\Psi(e)=\frac{1}{2}\int_\Omega\left\{Ce:e-\left[\left(\sqrt{Ce:e}-\gamma\right)^+\right]^2\right\}dx, \quad \forall e\in S,
$$
respectively, where $(g)^+$ denotes the positive part of a function $g$. 

From Theorem \ref{theorem_psi} $(i), (ii)$ we know that for any $\alpha\in(0,+\infty)$ the values $\psi(\alpha)$, $\psi_h(\alpha)$ give a lower bound of $\zeta_{lim}$, and $\zeta_{lim,h}$, respectively. Since $\Phi$ is coercive on $\mathbb{R}^{3\times3}_{sym}$, it holds that $\zeta_{lim,h}\rightarrow\zeta_{lim}$ as $h\rightarrow0_+$ using Theorem \ref{th_lim_conv}.

For purposes of the experiment, we choose $\gamma = 10$ and the increments $\triangle\alpha$ defined as follows: $\triangle\alpha=20$ for $\alpha\in[0,2000]$ and $\triangle\alpha=100$ for $\alpha=[2000,10000]$. The path-following procedure has been terminated if $\alpha\geq 10000$.

The comparison of the loading paths for seven different meshes is shown in Figure \ref{fig_load_path2}. Since the curves practically coincide the zoom is depicted in Figure \ref{fig_load_path2_detail}. We see that the value $\zeta\approx 9.48$ turns out to be a suitable lower bound of $\zeta_{lim}$. Further, one can see that $\psi_h\leq\psi_{h'}$ for $h\leq h'$. Therefore one can expect uniform convergence of $\{\psi_h\}$ to $\psi$ on closed and bounded intervals using Dini's theorem.
\begin{figure}[htbp]
        \begin{center}
          \includegraphics[width=0.7\textwidth]{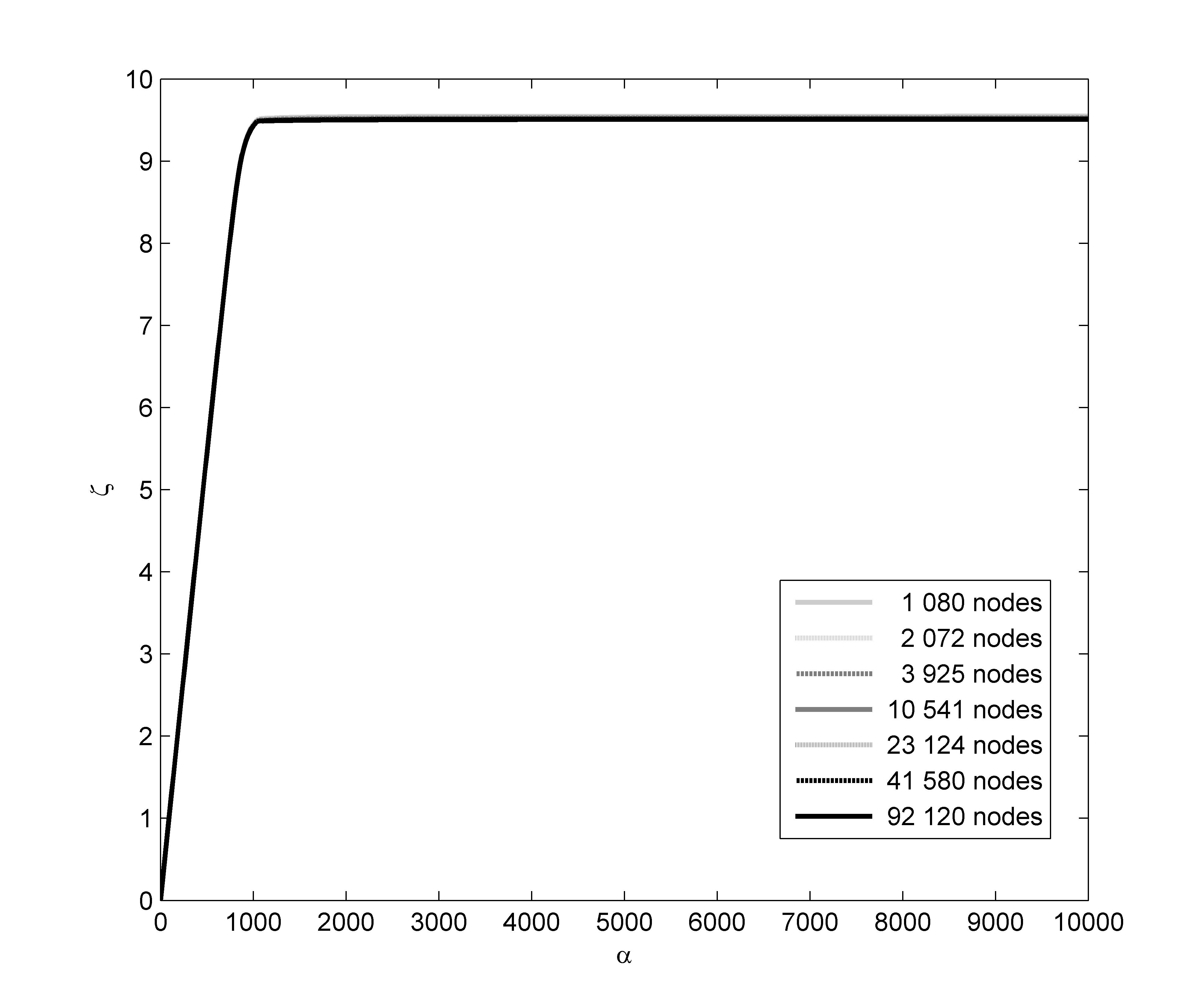}
        \end{center}
\vspace*{-1cm}
\caption{Loading paths up to $\alpha\in[0,10000]$.}
\label{fig_load_path2}
\end{figure}
\begin{figure}[htbp]
        \begin{center}
          \includegraphics[width=0.7\textwidth]{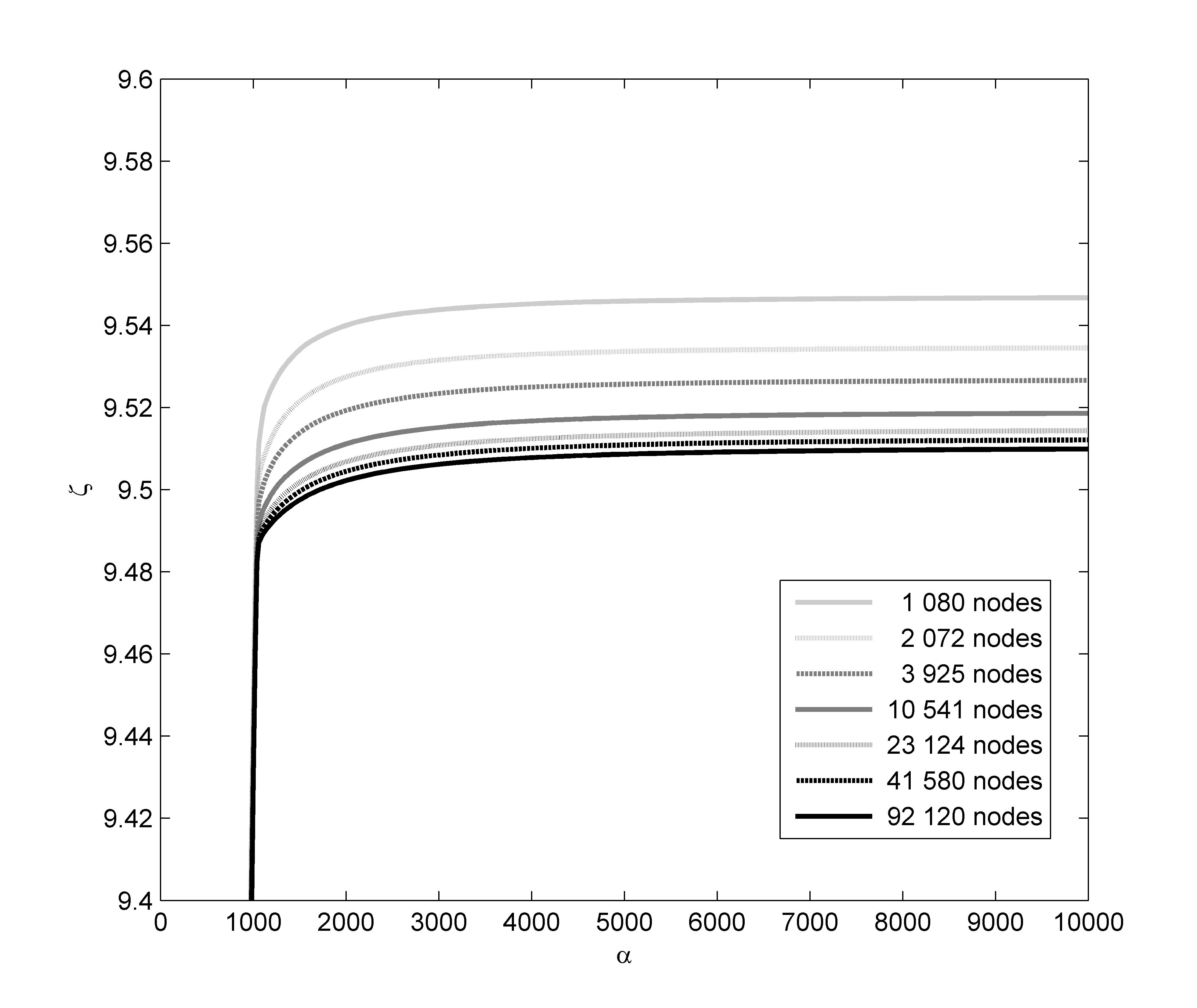}
        \end{center}
\vspace*{-1cm}
\caption{Loading paths for $\alpha\in[0,10000]$  (zoom).}
\label{fig_load_path2_detail}
\end{figure}


\subsection{Yield function 2 - von Mises criterion}

The von  Mises criterion \cite{T85, Ch96, SHHC13, CHKS14} is suitable for an isotropic and pressure insensitive material. The corresponding yield function has the form
\begin{equation}
\Phi(\tau)={\tau^D:\tau^D}, \quad \tau\in\mathbb{R}^{3\times3}_{sym},
\label{Phi_Mises}
\end{equation}
where $\tau^D = \tau -1/3\, \mbox{tr}{(\tau)}\iota$ is the deviatoric part of $\tau$.
If the elasticity tensor $C$ is defined as in (\ref{hooke_law}), then 
\begin{equation*}
\Psi(e):=\int_\Omega\left\{\frac{1}{2}Ce:e-\frac{1}{4\mu}\left[\left(2\mu\sqrt{e^D:e^D}-\gamma\right)^+\right]^2\right\} dx.
\end{equation*} 

Unlike Yield function 1, $\Phi$ defined by (\ref{Phi_Mises}) is not coercive on $\mathbb{R}^{3\times3}_{sym}$. Therefore convergence $\zeta_{lim,h}\rightarrow\zeta_{lim}$ as $h\rightarrow0_+$ is not guaranteed.


We choose $\gamma=450\sqrt{2/3}$ and $\triangle\alpha=5,\,100,\,1000$ for $\alpha\in[0,300],\,[300,10000], \,[10000, 100000]$, respectively.
The comparison of the loading paths for  seven different meshes is shown in  in Figure \ref{fig_detail}.  The curves practically coincide up to $\zeta=1$. Therefore the value $\zeta=1$ seems to be a reliable lower estimate of $\zeta_{lim}$.  As in the previous example, one can see that $\psi_h\leq\psi_{h'}$ for $h\leq h'$. 
\begin{figure}[htbp]
        \begin{center}
          \includegraphics[width=0.7\textwidth]{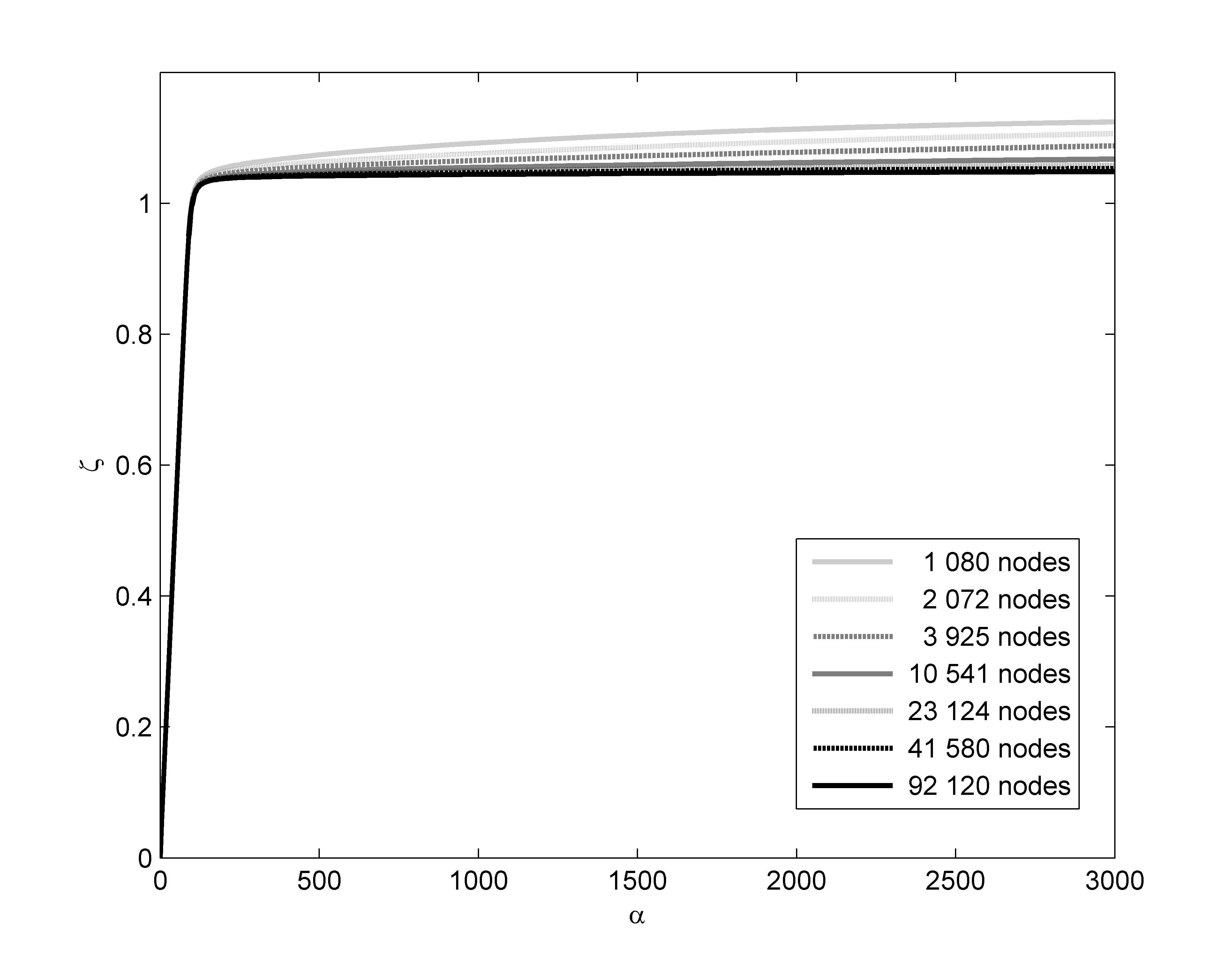}
        \end{center}
\vspace*{-1cm}
\caption{Loading paths up to $\alpha=3000$.}
\label{fig_detail}
\end{figure}

In Figure \ref{fig_load_path}, zooms of the loading paths up to $\alpha=100000$ for the seven meshes are displayed. We observe that the curve representing the coarsest mesh is almost constant in a vicinity of $\alpha=100000$ and the corresponding value of $\psi_h$ is approximately equal to 1.14 there. So one can expect that $\zeta_{lim}\in[1.00,1.14]$. On the other hand, pointwise convergence of $\{\psi_h(\alpha)\}$ becomes slow for large values of $\alpha$. Therefore, direct convergence $\zeta_{lim,h}\rightarrow\zeta_{lim}$ as $h\rightarrow0_+$ seems to be at least problematic.
\begin{figure}[htbp]
        \begin{center}
          \includegraphics[width=0.7\textwidth]{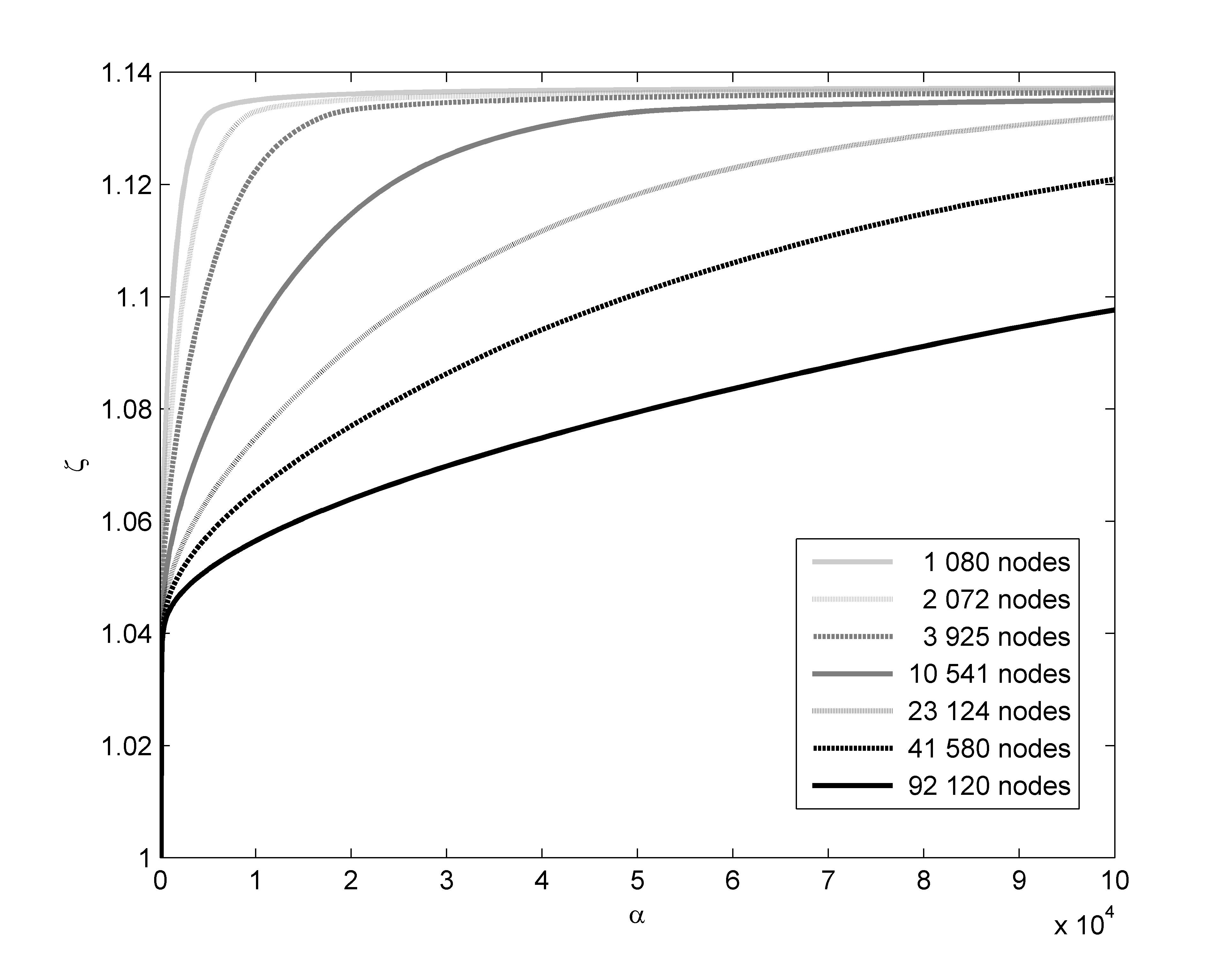}
        \end{center}
\vspace*{-1cm}
\caption{Loading paths up to $\alpha=100000$   (zoom).}
\label{fig_load_path}
\end{figure}

\section{Conclusion}
\label{sec_conclusion}
The paper deals with a new incremental method for computing the limit load in deformation plasticity models.  This procedure is based on a continuation parameter $\alpha$ ranging in $(0, +\infty)$ which is dual to the standard loading parameter $\zeta\in(0,\zeta_{lim})$, where $\zeta_{lim}$ is the critical value of $\zeta$. We have shown that there exists a continuous, nondecreasing function $\psi$ in $(0,+\infty)$ and such that $\psi(\alpha)\rightarrow\zeta_{lim}$ if $\alpha\rightarrow+\infty$. Therefore $\psi(\alpha)$ gives a guaranteed lower bound of $\zeta_{lim}$  for any $\alpha\in(0,+\infty)$. To evaluate $\psi(\alpha)$ for given $\alpha$ we derived a minimization problem for the stored energy functional subject to the constraint $L(v)=\alpha$ whose solutions define the respective value $\psi(\alpha)$. The second part of the paper was devoted to a finite element discretization  and convergence analysis. In particular, convergence of the discrete loading parameters $\zeta_{lim,h}$ to $\zeta_{lim}$ as $h\rightarrow0_+$ was proved for some yield functions. Finally, numerical experiments confirmed the efficiency of the proposed method.

\section*{Acknowledgements}
This work was held in the frame of the scientific cooperation between the Czech and Russian academies of sciences - Institute of Geonics
CAS and St. Petersburg Department of Steklov Institute of Mathematics and supported by the European Regional Development Fund in the IT4Innovations Centre of Excellence project (CZ.1.05/1.1.00/02.0070). The third author (S.S.) acknowledges the support of the project 13-18652S (GA CR) and the institutional support reg. no. RVO:68145535.


\end{document}